\title{\vspace{-0.6cm} On subsets of the hypercube with prescribed Hamming distances}

\documentclass[11pt]{article}

\usepackage{amsmath, amssymb, amsthm, url, tikz, verbatim}

\newtheorem{theorem}{Theorem}[section]

\newtheorem{lemma}[theorem]{Lemma}

\newtheorem{corollary}[theorem]{Corollary}

\usetikzlibrary{calc}
\usetikzlibrary{patterns}
\usetikzlibrary{decorations.markings}
\usetikzlibrary{arrows,shapes.geometric,%
	decorations.pathreplacing,shapes,shadows}

\date{}

\author{
Hao Huang \thanks{
Department of Math and CS, Emory University, Atlanta, USA. Research supported in part by the Collaboration Grants from the Simons Foundation.
}
\and
Oleksiy Klurman \thanks{Department of Mathematics, KTH Royal Institute of Technology, Stockholm, Sweden.}
\and
Cosmin Pohoata \thanks{Department of Mathematics, California Institute of Technology, Pasadena, USA.}
}

\begin{document}
\maketitle
\abstract A theorem of Kleitman in extremal combinatorics states that a collection of binary vectors in $\{0, 1\}^n$ with diameter $d$ has cardinality at most that of a Hamming ball of radius $d/2$. In this paper, we give an algebraic proof of Kleitman's Theorem, by carefully choosing a pseudo-adjacency matrix for certain Hamming graphs, and applying the Cvetkovi\'c bound on independence numbers. This method also allows us to prove several extensions and generalizations of Kleitman's Theorem to other allowed distance sets, in particular blocks of consecutive integers that do not necessarily grow linearly with $n$. We also improve on a theorem of Alon about subsets of $\mathbb{F}_{p}^{n}$ whose difference set does not intersect $\left\{0,1\right\}^{n}$ nontrivially.

\section{Introduction}

A rough version of the isodiametric inequality (see \cite{isodiametric}) states that in $\mathbb{R}^n$, among all bodies of a given diameter, the $n$-dimensional ball has the largest volume. Analogues of the isodiametric inequality are considered in the discrete settings. Recall that the Hamming distance $d(\vec{x}, \vec{y})$ between two vectors $\vec{x}$ and $\vec{y}$ is the number of coordinates in which they differ. Solving a conjecture of Erd\H{o}s, in \cite{kleitman} Kleitman proved the following important theorem in extremal set theory. It can also be viewed as an isodiametric inequality for discrete hypercubes.

\begin{theorem}\label{thm_kleitman}
Suppose $\mathcal{F}$ is a collection of binary vectors in $\{0, 1\}^n$, such that the Hamming distance between any two vectors is at most $d < n$. Then 
$$|\mathcal{F}| \le 
\begin{cases}
{n \choose 0}+\binom{n}{1} + \cdots + \binom{n}{t},~~~~~\textup{for}~d=2t;\vspace{0.4cm}\\
2\left(\binom{n-1}{0}+ \cdots + \binom{n-1}{t}\right),~~~\textup{for~}d=2t+1.
\end{cases}.
$$
\end{theorem}
Both inequalities are sharp. When $d=2t$, the upper bound is attained by a Hamming ball of radius $d/2$, for instance, a collection of binary vectors having at most $t$ $1$-coordinates. For $d=2t+1$, an optimal example is given by the Cartesian product of $\{0,1\}$ and the $(n-1)$-dimensional Hamming ball of radius $t$, or alternatively, a Hamming ball of radius $d/2$ centered at $(1/2, 0, \cdots, 0)$. Kleitman's Theorem can be reduced to the celebrated Katona intersection theorem \cite{katona}, using the squashing operation (see \cite{a_katona} or \cite{frankl-tokushige}). Generalizations of this theorem have been studied for the $n$-dimensional grid $[m]^n$ with Hamming distance \cite{a_kh, frankl-furedi}; as well as $[m]^n$ and the $n$-dimensional torus $\mathbb{Z}_m^n$ with Manhattan distance \cite{acz, bollobas-leader, du-kleitman}.

In this paper, we give a new proof of Kleitman's Theorem, based on the following bound by Cvetkovi\'c: the independence number $\alpha(G)$ is bounded from above by the number of non-negative (resp. non-positive) eigenvalues of the adjacency matrix of $G$. This result, when extended to pseudo-adjacency matrices, is still correct. A careful choice of a proper pseudo-adjacency matrix would lead to an algebraic proof of Kleitman's isodiametric theorem. This method also allows us to prove a similar upper bound, when the allowed Hamming distances between two vectors is in a set of consecutive integers. In particular, we show the following new estimate.

\begin{theorem}\label{thm_st}
For given integers $t > s \ge 0$, suppose $\mathcal{F}$ is a collection of binary vectors in $\{0, 1\}^n$, such that for every $\vec{x}, \vec{y} \in \mathcal{F}$, $d(\vec{x}, \vec{y}) \in \mathcal{L}$, with $\mathcal{L}=\{2s+1, \cdots, 2t\}$, then for $n$ sufficiently large, 
$$|\mathcal{F}| \le \binom{n}{t-s}+\binom{n}{t-s+1}+ \cdots + \binom{n}{0}.$$
Similarly, if $\mathcal{L} = \{2s+1, \cdots, 2t+1\}$, then $|\mathcal{F}| \le (2+o(1))\binom{n}{t-s}$.
\end{theorem}

Subsets of the hypercube with Hamming distances in a prescribed set of consecutive integers appear in the coding theory literature in the regime when $s$ and $t$ are linear in $n$, for instance in the context of $\epsilon$-balanced codes (of length $n$). These are subsets $\mathcal{F}$ of $\left\{0,1\right\}^{n}$ with pairwise Hamming distances between $\frac{1-\epsilon}{2} n$ and $\frac{1+\epsilon}{2}n$. By mapping them to vectors on the unit sphere in $\mathbb{R}^{n}$ via
$$(v_{1},\ldots,v_{n}) \mapsto \frac{1}{\sqrt{n}} \cdot ( (-1)^{v_{1}},(-1)^{v_{2}},\ldots,(-1)^{v_{n}} ) \in \left\{- \frac{1}{\sqrt{n}},\frac{1}{\sqrt{n}}\right\}^{n},$$
one can easily note that in this case estimating $|\mathcal{F}|$ amounts to estimating the length of a certain spherical code, for which other methods are useful. We refer to \cite{alon} for more details. For our general range (in particular when $s$ and $t$ are small compared to $n$), the problem of upper bounding $|\mathcal{F}|$ is of a different nature, and results about spherical codes do not apply.

In Section \ref{sec_ext}, we discuss several such extensions of Kleitman's Theorem to other distance sets. A number of other techniques including the Croot-Lev-Pach lemma establishes asymptotically sharp bounds for these problems.  

In Section \ref{sec_intersective}, we consider a somewhat different extremal set theory problem. A set $H \subset \mathbb{Z}^+$ is {\it intersective} if whenever $A$ is a subset of positive upper density of $\mathbb{Z}$, we have $(A-A) \cap H \neq \emptyset$. In the late 1970s, S\'ark\"ozy \cite{sarkozy}, and independently Furstenberg \cite{furstenberg1, furstenberg2}, proved that the set of perfect squares is intersective. A quantitative version has also been considered. Denote by $D(H, N)$ the maximum size of a subset $A \subset \{1, \cdots, N\}$ such that $(A-A) \cap H =\emptyset$. It is not hard to see that a set is intersective if and only if $D(H, N)=o(N)$. For many intersective $H$, the asymptotics of $D(H, N)$ have been studied, here we refer the readers to a survey of L\^e \cite{lth}. One particular interesting extremal problem is the analogue of this notion on vector spaces over finite fields. Consider a $N$-dimensional lattice $\mathbb{F}_p^N$, with $\mathbb{F}_p$ being the finite field of $p$ elements for some prime $p$. Let $J=\{0, 1\}^N$ and $D_{\mathbb{F}_p}(J, N)$ be the maximum cardinality of $H \subset \mathbb{F}_p^N$ such that $(H-H) \cap J= \emptyset$. An use of Sperner's Theorem shows that $D_{\mathbb{F}_p}(J, N)=o(p^N)$ and Alon (see \cite{lth}) proved the following bounds (the second inequality being an instance of the polynomial method):
$$ \frac{(p-1)^N}{p\sqrt{N}} \ll D_{\mathbb{F}_p}(J, N) \le (p-1)^N.$$
In this paper we also use the spectral method to slightly improve this upper bound.
\begin{theorem}\label{thm_intersective}
$$D_{\mathbb{F}_p}(J, N) \le \left(1-\frac{1}{2}\left(1-\frac{1}{p-1}\right)^p\right)(p-1)^N.$$
\end{theorem}

\section{An algebraic proof of Kleitman's Theorem}
To put things into modern context, we start this section with an algebraic argument that comes very close to proving Theorem \ref{thm_kleitman}, but only ends up giving a weaker bound. The main idea is to use the following lemma of Croot, Lev, and Pach \cite{clp}.

\begin{lemma}\label{lem_clp}
Let $P \in \mathbb{F}_2[x_1, \cdots, x_n]$ be a multilinear polynomial of degree at most $d$, and let $M$ denote the $2^n \times 2^n$ matrix with entries $M_{\vec{x}, \vec{y}}=P(\vec{x}+\vec{y})$ for $\vec{x}, \vec{y} \in \mathbb{F}_2^n$. Then 
$$\operatorname{rank}_{\mathbb{F}_2}(M) \le 2 \sum_{i=0}^{\lfloor d/2 \rfloor} \binom{n}{i}.$$ 
\end{lemma}

This was the main ingredient in the recent cap-set problem breakthrough \cite{eg} and the driving force behind many recent developments in additive combinatorics. We refer the reader to \cite{clp} for the original application for which it was developed and to \cite{ls} for a better account of its recent history and a comprehensive list of references.

For this incipient discussion, we only address the case $d=2t$. We enumerate the elements of $\mathbb{F}_{2}^{n}$ and consider the $2^{n} \times 2^{n}$ matrix $M$ defined by
$$M_{\vec{x}, \vec{y}} = {d(\vec{x},\vec{y})-1 \choose 2t}:=\frac{(d(\vec{x},\vec{y})-1) \cdots (d(\vec{x},\vec{y})-2t)}{(2t)!}$$
for every $\vec{x}, \vec{y} \in \mathbb{F}_{2}^{n}$. Let $M'$ denote the $2^{n} \times 2^{n}$ binary matrix obtained from $M$ by reducing each element modulo $2$. Note that every two distinct vectors $\vec{x}, \vec{y} \in \mathcal{F}$ has Hamming distance in $\{1, \cdots, 2t\}$, and $\binom{z-1}{2t}$ equals $0$ for $z \in \{1, \cdots, 2t\}$, and non-zero for $z=0$. Therefore the matrix $M'$ restricts on $\mathcal{F} \times \mathcal{F}$ to a full-rank sub-matrix, and thus $\operatorname{rank} M \geq \operatorname{rank} M' \geq |\mathcal{F}|$. On the other hand, there's a polynomial $p \in \mathbb{F}_2[t_1,..,t_n]$ with $\deg p \leq 2t$ so that 
$$p(\vec{x}-\vec{y}) = {d(\vec{x}, \vec{y})-1 \choose 2t} \mod 2, \ \ \ \text{for every $\vec{x}, \vec{y} \in \mathbb{F}_{2}^{n}$}$$ 
This polynomial is given explicitly by 
$$p(t_{1},\ldots,t_{n}) = \sum_{S \subset \left\{1,..,n\right\}, |S| \leq k} \prod_{i \in S} t_{i}.$$
Indeed, note that for every $x,y \in \mathbb{F}_{2}^{n}$,
$${d(\vec{x}, \vec{y})-1 \choose 2t} = \sum_{\ell=0}^{2t} (-1)^{\ell} {d(\vec{x}, \vec{y}) \choose \ell}.$$
Furthermore, in $\mathbb{F}_{2}$ we also have that
$$\sum_{|S| = \ell} {\prod_{i \in S} (x_i - y_i)} = {d(\vec{x}, \vec{y}) \choose \ell},$$
so
$${d(\vec{x}, \vec{y})-1 \choose 2t} = \sum_{S \subset \{1,..,n\}, |S| \leq 2t} (-1)^{|S|} \prod_{i \in S} (x_i - y_i) = \sum_{S \subset \{1,..,n\}, |S| \leq 2t} \prod_{i \in S} (x_i - y_i),$$
as claimed. Lemma \ref{lem_clp} then immediately implies
$$|\mathcal{F}| \leq 2\sum_{i=0}^{t} {n \choose i}.$$
When $d=2t+1$, it is not to hard to adapt the above argument to show that
$$|\mathcal{F}| \leq 4\left(\binom{n-1}{0}+ \cdots + \binom{n-1}{t}\right).$$

One can however improve on this rank argument and establish the precise version of Theorem \ref{thm_kleitman}. In fact, we will prove Theorem \ref{thm_st}, but to keep things simple for the rest of this section we will stick to the case when $s=0$ which recovers Theorem \ref{thm_kleitman}. We start with a few lemmas involving simple linear algebra. 

Let $M_{n, k}$ be a $2^n \times 2^n$ matrix, whose rows and columns are indexed by vectors in $\{0, 1\}^n$. The $(\vec{x}, \vec{y})$-th entry of $M_{n, k}$ is equal to $1$ if and only if $\vec{x}$ and $\vec{y}$ differ in exactly $k$ coordinates, and $0$ otherwise. For example, $M_{n, 1}$ is the adjacency matrix of the $n$-dimensional hypercube, and $M_{n, k}$ is the adjacency matrix of a Hamming-type graph in which two vertices are adjacent if they are at distance $k$. The following lemma determines the spectrum of all $M_{n, k}$ for all $1 \le k \le n$.
\begin{lemma}\label{lemma_eig1}
The spectrum of $M_{n, k}$ consists of $K_k(i; n)$ with multiplicity $\binom{n}{i}$, for $i=0, \cdots, n$. Here $K_k(i; n)$ is the Krawtchouk polynomial with parameter $2$: 
$$K_k(i; n) = \sum_{j=0}^k (-1)^j \binom{i}{j} \binom{n-i}{k-j}.$$
\end{lemma}
For example, when $k=1$, it is easy to check that the eigenvalues of the $n$-dimensional hypercube are $K_1(i; n)=n-2i$ with multiplicity $\binom{n}{i}$. The lemma can be found in \cite{lint-wilson} (Theorem 30.1). For completeness, we include its proof below using the Fourier transform on hypercubes as eigenvectors. Throughout the proof we use the notation $d(U, V)$ for the Hamming distance between the indicator vectors of $U$ and $V$, for two subsets $U, V \subset [n]$.

\begin{proof}
Let $\vec{v}_S$ be a vector in $\mathbb{R}^{2^n}$ defined as (with its $2^n$ coordinates viewed as subsets of $[n]$):
$$(\vec{v}_S)_T=(-1)^{|S \cap T|}.$$
It is not hard to show that $\{\vec{v}_S\}_{S \subset [n]}$ form an orthogonal basis. On the other hand,
\begin{align*}
(M_{n, k} \vec{v}_S)_T&=\sum_{U \subset [n]} (M_{n, k})_{T, U} (\vec{v}_S)_U=\sum_{U: d(U, T)=k} (-1)^{|S \cap U|}.
\end{align*}
Note that the number of sets $U$ with the property that $U$ and $T$ differ in $j$ coordinates in $S$ is equal to $\binom{|S|}{j} \binom{n-|S|}{k-j}$. For each of such $U$, 
$$(-1)^{|S \cap U|}=(-1)^{|S \cap T|} \cdot (-1)^j,$$
since $|S \cap U|=|S \cap T|+|(T \Delta U) \cap S|-2|S \cap T \cap \overline{U}|$. Therefore
\begin{align*}
(M_{n, k} \vec{v}_S)_T=(-1)^{|S \cap T|} \cdot \sum_{j=0}^{|S|} (-1)^{j} \binom{|S|}{j} \binom{n-|S|}{k-j} =K_{k}(|S|; n) (\vec{v}_S)_T.
\end{align*}
This immediately shows that $K_k(i; n)$ are eigenvalues of $M_{n, k}$ with multiplicity $\binom{n}{i}$.
\end{proof}
From the proof of Lemma \ref{lemma_eig1}, observe that for fixed $n$, the eigenspace decomposition of $M_{n, k}$ is the same for every $k$. Hence it is straightforward to establish the following result.
\begin{lemma}\label{lemma_eig2}
Suppose $f(1), \cdots, f(n)$ is a sequence of real numbers and let $M=\sum_{k=1}^n f(k) M_{n, k}$. Then the spectrum of $M$ consists of 
$$\lambda_i=\sum_{k=1}^n f(k) K_k(i; n)$$
with multiplicity $\binom{n}{i}$, for $i=0, \cdots, n.$
\end{lemma}

The following well-known theorem studies the relation between the spectrum of a symmetric matrix and that of its principal minor.
\begin{lemma} (Cauchy's Interlacing Theorem)
Let $A$ be a symmetric matrix of size $n$, and $B$ is a principal minor of $A$ of size $m \le n$. Suppose the eigenvalues of $A$ are $\lambda_1 \ge \lambda_2 \ge \cdots \ge \lambda_n$, and the eigenvalues of $B$ are $\mu_1 \ge \cdots \ge \mu_m$. Then for $1 \le i \le m$, we have
$$\lambda_{i+n-m} \le \mu_i \le \lambda_i.$$
\end{lemma}
The following corollary of the Cauchy's Interlacing Theorem was discovered earlier by Cvetkovi\'c \cite{cvetkovic}. It provides a useful technique to bound the independence number of a graph.
\begin{corollary}\label{cor}
Let $G$ be a $n$-vertex graph, and $M$ be a symmetric $n \times n$ matrix such that $M_{ij}=0$ whenever $ij \not\in E(G)$ (such $M$ is often called a {\it pseudo-adjacency matrix} of $G$). Let $n_{\le 0}(M)$ (resp. $n_{\ge 0}(M)$) be the number of non-positive (resp. non-negative) eigenvalues of $M$.  Then the independence number of $G$ satisfies 
$$\alpha(G) \le \min \{n_{\le 0}(M), n_{\ge 0}(M) \}$$
\end{corollary}
\begin{proof}
Suppose $I$ is a maximum independent set of $G$ with $|I|=\alpha(G)$. Then $I$ naturally corresponds to an all-zero principal minor $B$ of $M$. And the eigenvalues of $B$ are $\mu_1 = \cdots = \mu_{\alpha(G)}=0$. Let $\lambda_1 \ge \cdots \ge \lambda_n$ be the eigenvalues of $M$. By Cauchy's Interlacing Theorem, 
$$0=\mu_{\alpha(G)} \le \lambda_{\alpha(G)}.$$
So $M$ has at least $\alpha(G)$ non-negative eigenvalues. Similarly,
$$\lambda_{1+n-\alpha(G)}\le \mu_{1}=0,$$
which implies that $M$ has at least $\alpha(G)$ non-positive eigenvalues.
\end{proof}
~\\
Now we are ready to prove Theorem \ref{thm_kleitman}.
\begin{proof}[Proof of Theorem \ref{thm_kleitman}]
For given $n, d$, we define a graph $G$ whose vertex set $V(G)=\{0,1\}^n$, and two vertices are adjacent if their Hamming distance is at least $d+1$. Kleitman's problem is now equivalent to determining the independence number $\alpha(G)$.
	
We start with the even case $d=2t$. By Corollary \ref{cor} applied to $G$, it suffices to find real numbers $f(k)$ for $k=2t+1, \cdots, n$ and define $M=\sum_{k=2t+1}^{n} f(k) M_{n, k}$, such that either the number of non-positive or non-negative eigenvalues of $M$ is at most $\sum_{i=0}^t \binom{n}{i}$. 

At this point, it is perhaps important to mention that choosing $f(k) = {k - 1 \choose t}$ recovers the $2^{n} \times 2^{n}$ symmetric matrix $M$ defined by $M_{x,y} = {d(x,y)-1 \choose 2t}$ from the Croot-Lev-Pach approach, however this is {\it{not}} going to be the choice we are going to make for the sequence $f(1),\ldots,f(n)$. We choose $f(k)=\binom{\ell}{t}$ if $k=2\ell+1$ or $k=2\ell+2$. Equivalently $f(k)=\binom{\lfloor (k-1)/2 \rfloor}{t}$. By Lemma \ref{lemma_eig2}, the eigenvalue of $M$ with multiplicity $\binom{n}{i}$ is equal to 
\begin{align*}
\lambda_i=\sum_{k=2t+1}^{n} f(k) \sum_{j=0}^k (-1)^j \binom{i}{j} \binom{n-i}{k-j},
\end{align*}
for $i=0, \cdots, n$. Although computing the exact value of $\lambda_i$'s might not be easy, it turns out that we can determine their signs in a rather straightforward way. We claim that for every $t$, we have 
\begin{itemize}
\item $(-1)^i\lambda_i > 0$ for $i=0, \cdots, t$.
\item $\lambda_{n-i}=\lambda_{i+1}$ for $i=0, \cdots, t-1$.
\item $\lambda_{t+1}=\lambda_{t+2} = \cdots = \lambda_{n-t} = (-1)^{t+1}$.
\end{itemize}
To show the above claims, we use generating functions and observe that $\lambda_i$ is equal to the constant term of the following formal power series:
$$\left(\sum_{k=2t+1}^n f(k)x^{-k} \right)\left(\sum_{j=0}^i(-1)^j \binom{i}{j} x^j\right)\left(\sum_{\ell=0}^{n-i} \binom{n-i}{\ell} x^{\ell}\right).$$
Here in the generating function, we may extend the sum and the domain of $f$ to all the integers greater or equal to $2t+1$, with $f(k)=\binom{\lfloor (k-1)/2 \rfloor}{t}$ as before. This would not affect the constant term since for $\binom{i}{j}$ and $\binom{n-i}{\ell}$ to be non-zero, one must have $j \le i$ and $\ell \le n-i$. So $f(k)$ for only those $k$ up to $n$ may contribute to the constant term.

A quick calculation shows that
\begin{align*}
\sum_{k=2t+1}^{\infty} f(k)x^{-k}&=\binom{t}{t}(x^{-(2t+1)}+x^{-(2t+2)})+ \binom{t+1}{t}(x^{-(2t+3)}+x^{-(2t+4)})+\cdots\\
&=x^{-(2t+1)}(1+x^{-1})\left(\binom{t}{t}+\binom{t+1}{t}x^{-2}+\cdots\right)\\
&=x^{-(2t+1)}(1+x^{-1})(1-x^{-2})^{-(t+1)}\\
&=\frac{x+1}{(x^2-1)^{t+1}},
\end{align*}
in which the power series converges when $|x|>1$.

Note that 
$$\sum_{j=0}^n(-1)^j \binom{i}{j} x^j=(1-x)^i ~~~\textup{and}~~~\sum_{\ell=0}^{n-i} \binom{n-i}{\ell} x^{\ell}=(1+x)^{n-i}.$$
Therefore, $\lambda_i$ is equal to the constant term of the following power series:
\begin{align*}
\frac{x+1}{(x^2-1)^{t+1}} \cdot (1-x)^i (1+x)^{n-i} =(-1)^{t+1} (1+x)^{n-i-t} (1-x)^{i-t-1}.
\end{align*}
For $t+1 \le i \le n-t$, both $n-i-t$ and $i-t-1$ are nonnegative, so the constant term is equal to $(-1)^{t+1}$. For $0 \le i \le t$, one needs to consider the constant term of 
\begin{align*}
(-1)^{t+1}\frac{(1+x)^{n-i-t}}{(1-x)^{t+1-i}} = (-1)^{i}\frac{(1+x)^{n-i-t}}{(x-1)^{t+1-i}}=(-1)^i x^{-(t+1-i)}(1+x)^{n-i-t}(1-\frac{1}{x})^{-(t+1-i)}
\end{align*}
Obviously in the expansion of $(1+x)^{n-i-t}(1-\frac{1}{x})^{-(t+1-i)}$ for $x>1$, all the coefficients are positive. So $(-1)^i \lambda_i$ is positive since $t+1-i\le n-i-t$. For $i>n-t$, note that in a power series, substituting $x$ by $-x$ does not change the constant term. Therefore letting $i=n+1-j$, $\lambda_i$ is equal to  the constant term of 
$$(-1)^{t+1}\frac{(1+x)^{i-(t+1)}}{(1-x)^{i-(n-t)}}=(-1)^{t+1} \frac{(1+x)^{(n-t)-j}}{(1-x)^{(t+1)-j}},$$
which is exactly $\lambda_j=\lambda_{n+1-i}$.

Therefore for even $t=2m$, the only non-negative eigenvalues are $\lambda_0, \lambda_2, \cdots, \lambda_{2m}$, $\lambda_{n-1}, \lambda_{n-3},\\ \cdots, \lambda_{n-(2m-1)}$, and their  multiplicities add up to 
$\sum_{i=0}^t \binom{n}{i}$. Similarly when $t=2m+1$, the only non-positive eigenvalues are 
$\lambda_1, \cdots, \lambda_{2m+1}, \lambda_{n}, \cdots, \lambda_{n-2m}$, and their total multiplicity equals $\sum_{i=0}^t \binom{n}{i}$ as well. This finishes the proof for the even case.

The proof for the odd case $d=2t+1$ works in a similar fashion, except that we have to choose $f(k)$ for $k=2t+2, \cdots, n$ in a slightly different way. Here we define $f(k)=0$ for odd $k$, and $f(k)=\binom{k/2-1}{t}$ for even $k$. By a similar argument, the eigenvalue $\lambda_i$ with multiplicity $\binom{n}{i}$ is equal to the constant term of the following formal sum:
$$\left(\sum_{k=2t+2}^{\infty} f(k)x^{-k} \right)\left(\sum_{j=0}^i(-1)^j \binom{i}{j} x^j\right)\left(\sum_{l=0}^{n-i} \binom{n-i}{l} x^l\right).$$
It is equal to 
\begin{align*}
& ~~~x^{-(2t+2)}\left(\binom{t}{t}+\binom{t+1}{t}x^{-2}+\binom{t+2}{t}x^{-4}+ \cdots \right)(1-x)^i (1+x)^{n-i}\\
&=x^{-(2t+2)}(1-x^{-2})^{-(t+1)}(1-x)^i (1+x)^{n-i}\\
&=(-1)^{t+1}(1-x)^{i-t-1} (1+x)^{n-i-t-1}
\end{align*}
Once again, for $t+1 \le i \le n-t-1$, the constant term equals $(-1)^{t+1}$. For $0 \le i \le t$, it is equal to 
\begin{align*}
(-1)^{t+1} (1-x)^{-(t+1-i)}(1+x)^{n-i-t-1}&=(-1)^i x^{-(t+1-i)}(1+x)^{n-i-t-1}(1-\frac{1}{x})^{-(t+1-i)} 
\end{align*}
Again note that the expansions of both $(1+x)^{n-i-t-1}$ and $(1-\frac{1}{x})^{-(t+1-i)}$ only consist of positive coefficients. Therefore $(-1)^i \lambda_i >0$. Similar as before, one can show that for $n-t \le i \le n$, $\lambda_i=\lambda_{n-i}$. Now we apply Corollary \ref{cor} once again. Note that none of $\lambda_i$'s is zero. We only need to show that either the number of positive or negative eigenvalues is small. For even $t=2m$, the only positive $\lambda_i$ are $\lambda_0, \lambda_2, \cdots, \lambda_{2m}$ and $\lambda_n, \lambda_{n-2}, \cdots, \lambda_{n-2m}$, whose total multiplicity is equal to $2\sum_{i=0}^m \binom{n}{2i}$. For odd $t=2m+1$, the only negative eigenvalues are $\lambda_1, \lambda_3, \cdots, \lambda_{2m+1}$ and $\lambda_{n-1}, \lambda_{n-3}, \cdots, \lambda_{n-(2m+1)}$, whose multiplicity is $2\sum_{i=0}^m \binom{n}{2i+1}$. Finally, it is easy to check both sum equals the sum in Kleitman's Theorem for the case $d=2t+1$, noting that for even $t=2m$, 
\begin{align*}
2 \sum_{i=0}^m \binom{n}{2i}= 2 \left(\sum_{i=0}^m\left( \binom{n-1}{2i-1}+\binom{n-1}{2i}\right)\right)=2\sum_{i=0}^t \binom{n-1}{i},
\end{align*}
and for odd $t=2m+1$, 
\begin{align*}
2\sum_{i=0}^m \binom{n}{2i+1}=2\left(\sum_{i=0}^m \left(\binom{n-1}{2i} + \binom{n-1}{2i+1}\right)\right)=2 \sum_{i=0}^t \binom{n-1}{i}.
\end{align*}
\end{proof}
\section{Extensions to arbitrary distance sets}\label{sec_ext}
In this section, we discuss a few generalizations of Kleitman's theorem to other sets of allowed distances. The next theorem shows that a bound similar to Kleitman's holds for all $n$ when the set of allowed distances consists of consecutive integers.
\begin{theorem}\label{thm_st}
For given integers $t > s \ge 0$, suppose $\mathcal{F}$ is a collection of binary vectors in $\{0, 1\}^n$, such that for every $\vec{x}, \vec{y} \in \mathcal{F}$, $d(\vec{x}, \vec{y}) \in \mathcal{L}$, with $\mathcal{L}=\{2s+1, \cdots, 2t\}$, then for all $n$, 
$$|\mathcal{F}| \le \binom{n}{t-s}+2\binom{n}{t-s+1}+ \cdots + 2\binom{n}{0}.$$
\end{theorem}
\begin{proof}
We follow the proof of Theorem \ref{thm_kleitman} for a different pseudo-adjacency matrix $M=\sum_{k=1}^n f(k) M_{n, k}$. Here for every integer $\ell \ge 0$, we take
$$f(2 \ell+ 1)=f(2\ell+ 2)=\binom{\ell-s}{t-s}.$$
By extending the definition of binomial coefficients to the whole set of integers, we have that $f(k) \neq 0$ if $k \ge 2t+1$, or $1 \le k \le 2s$. Therefore $M$ is a pseudo-adjacency matrix for our purpose of bounding independence number. 

The remaining task is to calculate the eigenvalues of $M$. Using similar arguments as in Theorem \ref{thm_kleitman}, we have that $M$ has an eigenvalue $\lambda_i$ of multiplicity $\binom{n}{i}$, and $\lambda_i$ is equal to the constant term in the following formal power series, for $|x|>1$,
$$\left(\sum_{k=1}^{\infty} f(k)x^{-k} \right)(1-x)^i (1+x)^{n-i}.$$
Let $g_s(x)=\sum_{k=0}^{\infty} \binom{k-s}{t-s}x^{k}$, we will first show by induction (on $s$) that for $|x|<1$, it converges to 
$$h_s(x)=\frac{\sum_{j=s}^t \binom{t}{j} (x-1)^{t-j}}{(1-x)^{t-s+1}}.$$
For $s=0$, $$g_s(x)=\binom{t}{t} x^{t}+ \binom{t+1}{t}x^{t+1} + \cdots =x^t/(1-x)^{t+1},$$
which equals $h_s(x)$. Assume for $s \ge 0$, $g_s(x)=h_s(x)$, then 
\begin{align*}
h_{s+1}(x)&=h_s(x)(1-x)-(-1)^{t-s}\binom{t}{s} = g_s(x)(1-x)-(-1)^{t-s}\binom{t}{s}\\
&=\sum_{k=0}^{\infty} \binom{k-s}{t-s}x^k - \sum_{k=0}^{\infty}\binom{k-s}{t-s}x^{k+1} -(-1)^{t-s}\binom{t}{s}\\
&=\sum_{k=0}^{\infty} \binom{k-s}{t-s}x^k - \sum_{k=1}^{\infty}\binom{k-s-1}{t-s}x^{k} -(-1)^{t-s}\binom{t}{s}\\
&=\binom{-s-1}{t-s} + \sum_{k=0}^{\infty} \binom{k-s-1}{t-s-1}x^k - (-1)^{t-s} \binom{t}{s} = g_{s+1}(x).
\end{align*}
This completes the proof that $g_s(x)=h_s(x)$. Now we have
\begin{align}\label{eqn_long}
\left(\sum_{k=1}^{\infty} f(k)x^{-k} \right)(1-x)^i (1+x)^{n-i}&=
(1-x)^i (1+x)^{n-i} (x^{-1}+x^{-2}) g_s(x^{-2}) \nonumber\\
&=(1-x)^i (1+x)^{n-i+1}x^{-2} \frac{\sum_{j=s}^t \binom{t}{j} (x^{-2}-1)^{t-j}}{(1-x^{-2})^{t-s+1}}\\
&=(-1)^{t-s+1} \sum_{j=s}^t  \binom{t}{j}(1+x)^{n-i-j+s}(1-x)^{i-j+s-1} x^{2(j-s)}. \nonumber
\end{align}
Recall that the eigenvalue $\lambda_i$ of multiplicity $\binom{n}{i}$ is equal to the constant term in the power series. Note that the sum is over all integers $j$ between $s$ and $t$. In this range, $n-i-j+s \ge n-(t-s)-i$, $i-j+s-1 \ge i-1-(t-s)$, and $2(j-s)$ is strictly greater than $0$ except for $j=s$. So for $(t-s)+1 \le i \le n-(t-s)$, the product in the sum is a polynomial divisible by $x$ and thus its constant term is $0$. Only $j=s$ would contribute to the constant term. Therefore for $i$ in this range,
$$\lambda_i=(-1)^{t-s+1} \binom{t}{s}.$$
In other words, for $(t-s)+1 \le i \le n-(t-s)$, $\lambda_i$ have the same sign. This would immediately imply 
\begin{align*}
\alpha(G) &\le \min\{n_{\le 0}(M), n_{\ge 0}(M)\} \le \sum_{i=0}^{t-s} \binom{n}{i} + \sum_{i=n-(t-s)+1}^n \binom{n}{i}\\
&= \binom{n}{t-s}+2\binom{n}{t-s+1}+ \cdots + 2\binom{n}{0}.
\end{align*}
\end{proof}
\noindent \textbf{Remark.} For sufficiently large $n$, by a slightly more careful analysis, one can actually remove the factors of $2$ in the statement of Theorem \ref{thm_st}, and show that 
$$|\mathcal{F}| \le \binom{n}{t-s}+\binom{n}{t-s+1}+ \cdots + \binom{n}{0},$$
which gives the exact same bound as in Theorem \ref{thm_st}. To achieve this goal, one can show that when $t-s$ is odd, $\lambda_{2i} > 0$ whenever $0 \le 2i \le t-s$ and $\lambda_{n-2i-1}>0$ whenever $n-2i-1>n-(t-s)$; and when $t-s$ is even, $\lambda_{2i+1}<0$ whenever $0\le 2i+1 \le t-s$, and $\lambda_{n-2i}<0$ whenever $n-2i>n-(t-s)$. Then applying Corollary \ref{cor} gives the desire upper bound. The calculations are a bit tedious, so we decided to omit the details, since it still gives a upper bound that is asymptotically the same, $(1+o(1))\binom{n}{t-s}$. Moreover, using similar techniques, one can show that if the set of allowed distances are $\{2s+1, \cdots, 2t+1\}$, then $|\mathcal{F}| \le (2+o(1))\binom{n}{t-s}$, generalizing Kleitman's Theorem for odd diameters.\\

When $\mathcal{L}=\{2s+1,\cdots, 2t\}$, Theorem \ref{thm_st} gives an upper bound which is $O(n^{t-s})$ for fixed $s, t$ and large $n$. The following theorem shows that this upper bound is tight up to a constant factor.

\begin{theorem}\label{thm_lower}
For sufficiently large $n$, there exists a family $\mathcal{F}$ of $(1/\binom{t}{s}-o(1)) \binom{n}{t-s}$ binary vectors in $\{0, 1\}^n$, such that for every two vectors $\vec{x}, \vec{y} \in \mathcal{F}$, $d(\vec{x}, \vec{y}) \in \mathcal{L}=\{2s+1, \cdots, 2t\}$.
\end{theorem}
\begin{proof}
We will define a family $\mathcal{F}$ consisting of some vectors with $2t$ $1$-coordinates. For two such vectors $\vec{x}$ and $\vec{y}$, denote by $X$ and $Y$ the $t$-sets they naturally correspond to. Then $d(\vec{x}, \vec{y}) \in \{2s+1, \cdots, 2t\}$ is equivalent to $4t-2|X \cap Y| \in \{2s+1, \cdots, 2t\}$, i.e. $|X \cap Y| \in \{t, \cdots, 2t-s-1\}$. By the famous result of R\"odl  \cite{rodl} on the Erd\H os-Hanani Conjecture \cite{erdos-hanani}, for sufficiently large $n$, there exists a packing of $m=(1-o(1))\binom{n-t}{t-s}/\binom{t}{t-s}$ copies of complete $(t-s)$-uniform hypergraphs $K_{t}^{t-s}$ in $K_{n-t}^{t-s}$. Suppose the vertex set of these hypercliques are $V_1, \cdots, V_m$. Then $|V_i|=t$ and $|V_i \cap V_j| \in \{0, \cdots, t-s-1\}$. Take $F_i=V_i \cup \{n-t+1, \cdots, n\}$ and $\mathcal{F}=\{F_1, \cdots, F_m\}$. It is easy to check that $|F_i|=2t$ and $|F_i \cap F_j| \in \{t, \cdots, 2t-s-1\}$.
\end{proof}

For a set $\mathcal{L}$ of integers, let $f_{\mathcal{L}}(n)$ be the maximum number of binary vectors in $\{0, 1\}^n$ with pairwise Hamming distance in $\mathcal{L}$. The theorems above show that 
$$(1-o(1)) \binom{n}{t-s}/\binom{t}{s} \le f_{\{2s+1, \cdots, 2t\}}(n) \le (1+o(1))\binom{n}{t-s}.$$	
For $s=0$ the upper and lower bounds agree, as shown by Theorem \ref{thm_kleitman}. For general $s$ and $t$, it is plausible that the lower bound is asymptotically tight. We are able to verify this conjecture for the special case $\mathcal{L}=\{2s+1, 2s+2\}$. We start with the following lemma on subsets of restricted intersection sizes.
\begin{lemma} \label{lem_eq_int}
Given integers $i > j \ge 1$, suppose $\mathcal{F}$ is a collection of $i$-subsets of $[n]$, whose pairwise intersection has size exactly $j$, then $|\mathcal{F}| \le  (1+o(1))n/(i-j).$
\end{lemma}
\begin{proof}
Suppose $\mathcal{F}=\{F_1, \cdots, F_m\}$, and without loss of generality assume $F_1=\{1, \cdots, i\}$. For every $j$-subset $S$ of $[i]$, let $\mathcal{F}_S=\{F: F \cap [i]=S\}$, then by the assumption $\mathcal{F}=\{F_1\} \cup (\bigcup_S \mathcal{F}_S)$. Let $\mathcal{F}'_S=\{F \setminus S: F \in \mathcal{F}_S\}$. Then each non-empty $\mathcal{F}'_S$ consists of pairwise disjoint $(i-j)$-subsets of $\{i+1, \cdots, n\}$. This immediately gives $|\mathcal{F}_S| \le (n-i)/(i-j)$.

We claim that for two distinct $j$-sets $S$ and $T$, if both $\mathcal{F}_S$ and $\mathcal{F}_T$ are non-empty, then they both contain at most $i-j$ sets. This is because $|S \cap T|<j$ and thus $\mathcal{F}'_S$ and $\mathcal{F}'_T$ are cross-intersecting, and that a set $U \in \mathcal{F}'_S$ of size $i-j$ can only intersect with at most $i-j$ pairwise disjoint subsets. Therefore for all but at most one set $S$, $|\mathcal{F}_S| \le i-j$. Hence 
$$|\mathcal{F}| \le 1+\sum_{S: S \subset [i], |S|=j} |\mathcal{F}_S| \le 1+\left(\binom{i}{j}-1\right)(i-j)+\frac{n-i}{i-j}=(1+o(1))\frac{n}{i-j}.$$
This completes the proof.


\end{proof}
\begin{theorem}\label{thm_s+1}
For integers $s \ge 0$, 
$$f_{\{2s+1, 2s+2\}}(n) = (1+o(1)) \frac{n}{s+1}.$$
\end{theorem}
\begin{proof}
Let $\mathcal{F}$ be a family of $m$ vectors in $\{0, 1\}^n$ with pairwise Hamming distance either $2s+1$ or $2s+2$. Without loss of generality assume one of these vectors is the all-zero vector, then the remaining $m-1$ vectors are the indicator vectors of subsets of $[n]$ of size $2s+1$ or $2s+2$. Denote by $\mathcal{A}$ the family of these $(2s+1)$-sets, and $\mathcal{B}$ the family of $(2s+2)$-sets. For two sets $A_1, A_2 \in \mathcal{A}$, we have 
$$|A_1|+|A_2|-2|A_1 \cap A_2|=|A_1 \Delta A_2| \in \{2s+1, 2s+2\}.$$
By considering the parity, this gives $|A_1 \cap A_2|=s$. Similar arguments show that for two sets $B_1, B_2 \in \mathcal{B}$, $|B_1 \cap B_2|=s+1$. And for $A \in \mathcal{A}$, $B \in \mathcal{B}$, $|A \cap B|=s+1$. Now we construct a new family $\mathcal{A}'$ of subsets of $[n+1]$, by adding the element $n+1$ to each set in $\mathcal{A}$. It is straightforward to check that $\mathcal{C}=\mathcal{A}' \cup \mathcal{B}$ satisfies the property that every set contains $2s+2$ elements, while every two subsets intersect in exactly $s+1$ elements. Now applying Lemma \ref{lem_eq_int} for $\mathcal{C}$, we have $|\mathcal{C}| \le (1+o(1))n/(s+1)$, and the same upper bound on $|\mathcal{F}|$ and $f_{\{2s+1, 2s+2\}}(n)$ follows.

On the other hand, Theorem \ref{thm_lower} with $t=s+1$ gives $f_{\{2s+1, 2s+2\}}(n) \ge (1-o(1))n/(s+1)$ and this completes the proof.
\end{proof}

\smallskip

Note that $\{2s+1, \cdots, 2t\}$ is a set consisting of $2(t-s)$ integers. It is tempting to speculate that the order of magnitude of $f_{\mathcal{L}}(n)$ solely depends on the size of the set $\mathcal{L}$ of allowed distances. However this is false. For example, suppose $\mathcal{L}$ only consists of odd distances. Then $f_{\mathcal{L}}(n) \le 2$ since if the family contains three vectors, their corresponding subsets $A, B, C$ satisfy 
$$2(|A \cup B \cup C|-|A \cap B \cap C|)=|A \Delta B|+|A \Delta C|+|B \Delta C| \equiv 1 \pmod 2,$$
resulting in a contradiction. This observation immediately leads to the following simple upper bound for general $\mathcal{L}$.

\begin{theorem}\label{thm_even}
Let $\mathcal{L}$ be a set of distinct positive integers. Suppose $c$ of them are even numbers. Then when $n$ tends to infinity, 
$f_{\mathcal{L}}(n)=O(n^c).$
\end{theorem}
\begin{proof}
Suppose $\mathcal{L}=\{\ell_1, \cdots, \ell_s\}$, and $\mathcal{F}$ is a family of vectors in $\{0, 1\}^n$ with pairwise Hamming distances in $\mathcal{L}$. Without loss of generality assume $\vec{0} \in \mathcal{F}$, and the rest of the vectors correspond to subsets in a family $\mathcal{A}$. Then every subset in $\mathcal{A}$ has size in $\mathcal{L}=\{\ell_1, \cdots,  \ell_s\}$. Define $\mathcal{A}_i=\{A: |A|=\ell_i, A \in \mathcal{A}\}$, for $i=1, \cdots, s$. Then for two distinct subsets $X, Y$ in $\mathcal{A}_i$, their corresponding vectors have Hamming distance equal to 
$$2\ell_i - 2|X \cap Y|=|X|+|Y|-2|X \cap Y|=|X \Delta Y| \in \mathcal{L}.$$
Since there are $c$ even numbers in $\mathcal{L}$, $|X \cap Y|$ belongs to a set of at most $c$ possible intersection sizes. By the Frankl-Wilson Theorem \cite{fw}, $|\mathcal{A}_i| \le \binom{n}{c}+ \cdots + \binom{n}{0}$, and therefore
$$|\mathcal{F}|=1+|\mathcal{A}| =1+\sum_{i=1}^s |\mathcal{A}_i| =O(n^c).$$
\end{proof}

Although the problem of determining the order of magnitude for every fixed distance set $\mathcal{L}$ and sufficiently large $n$ seems beyond our reach, we can still establish asymptotically sharp bounds for some other special distance sets. In fact, we have already established one at the beginning of Section 2. Recall that we started by using the Croot-Lev-Pach Lemma to show a weaker version of Theorem \ref{thm_kleitman}. Similarly, we can also prove the following asymptotically sharp estimate for a different type of arithmetic constraint on $\mathcal{L}$.

\begin{theorem} 
For given integers $n, k$ such that $n \ge 2^k$, let $\mathcal{L}$ consist of all the integers between $1$ and $n$ that are not divisible by $2^k$. Then, for $n$ sufficiently large, we have
$$f_{\mathcal{L}}(n)=(2+o(1)) \binom{n}{2^{k-1}-1}.$$ 
\end{theorem}

The reader should compare this to Theorem \ref{thm_even}. This was also recorded independently by Ellenberg in \cite{jellenberg}. 

\begin{proof}
We start by proving the upper bound. Take a $2^n \times 2^n$ matrix $M$, whose rows and columns correspond to $n$-dimensional binary vectors, and $M_{\vec{x}, \vec{y}}=g(\vec{x}-\vec{y})$. Here $g: \mathbb{F}_2^n \rightarrow \mathbb{F}_2$ is the the following polynomial:
$$g(\vec{z})=\prod_{j=0}^{k-1} \left(1- {\|z\| \choose 2^j} \right).$$
Here $\| \cdot \|$ is the Hamming norm, so $d(\vec{x}, \vec{y})=\|\vec{x}-\vec{y}\|$. Suppose $\mathcal{F}$ is a family of vectors such that their pairwise Hamming distance is not divisible by $2^k$. Therefore for distinct $\vec{x}, \vec{y} \in \mathcal{F}$, in the binary representation of $\|\vec{x}-\vec{y}\|$, the last $k$ digits are not all $0$. 

At this point, we recall the classical Lucas' theorem.

\begin{lemma} \label{lucas} Given two positive integers $A \ge B$ and a prime number $p$, suppose their $p$-ary representation are $A=\sum_{i=0}^s a_i p^i$ and $B=\sum_{i=0}^s b_i p^i$, then 
	$$\binom{A}{B} \equiv \prod_{i=0}^s \binom{a_i}{b_i} \pmod p.$$
\end{lemma}

By Lemma \ref{lucas}, for some $j \in \{0, \cdots, k-1\}$, $\binom{\|z\|}{2^j} \equiv 1 \pmod 2$. Therefore $M_{\vec{x}, \vec{y}}=g(\vec{x}-\vec{y})\equiv 0 \pmod 2$. On the other hand, obviously $M_{\vec{x}, \vec{x}} \equiv 1 \pmod 2$. Therefore the family $\mathcal{F}$ naturally induces a submatrix of $M$, which is a unit matrix in $\mathbb{F}_2$ and has full rank. As a consequence, $|\mathcal{F}|$ is upper-bounded by the $\mathbb{F}_2$-rank of $M$. Note that $\deg(g)=\sum_{j=0}^{k-1} 2^j=2^k-1$. Lemma \ref{lem_clp} immediately implies
$$|\mathcal{F}| \le 2 \sum_{i=0}^{2^{k-1}-1} \binom{n}{i}.$$

The lower bound can be obtained again by the same extremal construction for Kleitman's theorem, when the allowed distance set is $\{1, \cdots, 2^{k}-1\}$. Theorem \ref{thm_kleitman} gives
$$f_{\mathcal{L}}(n) \ge 2 \sum_{j=0}^{2^{k-1}-1} \binom{n-1}{j}.$$ 
\end{proof}

\section{On intersective sets in $\mathbb{F}_p^N$} \label{sec_intersective}
In this section, we prove Theorem \ref{thm_intersective}. Here we briefly sketch the idea. We construct a graph $G$ with vertex set $\mathbb{F}_p^N$, two vertices $\vec{x}$ and $\vec{y}$ are adjacent if $\vec{x}-\vec{y}$ or $\vec{y}-\vec{x}$ is in $J=\{0, 1\}^N$. Then $\alpha(G)=D_{\mathbb{F}_p}(J, N)$. We will choose a pseudo-adjacency matrix for $G$ and apply Corollary \ref{cor}. 

The following lemma computes the spectrum of a family of matrices, that are natural candidates for the pseudo-adjacency matrix of $G$.

\begin{lemma}\label{lemma_eig}
Let $\omega=e^ {2i\pi/p}$ and  $M$ be a $p^N \times p^N$ matrix whose rows and columns are indexed by vectors in $\mathbb{F}_p^N$, and 
	$M_{\vec{u}, \vec{v}}=f(\vec{u}-\vec{v})$, where $f$ is a function mapping $\mathbb{F}_p^N$ to $\mathbb{R}$. Then the function $\chi_{\vec{v}}: \mathbb{F}_p^N \rightarrow \mathbb{C}$ with $\chi_{\vec{v}}(\vec{u}) =\omega^{\langle \vec{u}, \vec{v}\rangle}$, when viewed as a vector, is an eigenvector of $M$, corresponding to the eigenvalue
	$$\sum_{\vec{x}} f(\vec{x}) \omega^{-\langle \vec{v}, \vec{x}\rangle}.$$ Moreover all of them form a basis of $\mathbb{R}^{(p^N)}$.
\end{lemma}
\begin{proof}
	We first verify $\chi_{\vec{v}}$ is an eigenvector of $M$. We have
	\begin{align*}
	(M \chi_{\vec{v}})_{\vec{z}}&=\sum_{\vec{y}} M_{\vec{z}, \vec{y}}\cdot \chi_{\vec{v}}({\vec{y}}) = \sum_{\vec{y}} f(\vec{z}-\vec{y}) \cdot \omega^{\langle \vec{v},\vec{y} \rangle}\\
	&=\sum_{\vec{x}} f(\vec{x}) \omega^{\langle \vec{v}, \vec{z}-\vec{x} \rangle} = \omega^{\langle \vec{v}, \vec{z} \rangle} \cdot \sum_{\vec{x}} f(\vec{x}) \omega^{-\langle \vec{v}, \vec{x}\rangle}\\
	&=\chi_{\vec{v}}(\vec{z}) \cdot \left(\sum_{\vec{x}} f(\vec{x}) \omega^{-\langle \vec{v}, \vec{x}\rangle}\right)
	\end{align*}
	
	It is straightforward to show that $\chi_{\vec{v}}$ are linearly independent.
\end{proof}

Now we are ready to prove Theorem \ref{thm_intersective}.

\begin{proof}[Proof of Theorem \ref{thm_intersective}]
	From the discussions at the beginning of this section, we only need to upper bound the independence number of $G$.
	
	We define $M$ to be a $p^N \times p^N$ matrix with rows and columns indexed by vectors in $\mathbb{F}_p^N$. We let $M_{\vec{u}, \vec{v}}=(-1)^{c(\vec{u}-\vec{v})}$, for vectors $\vec{u} \neq \vec{v}$ with either $\vec{u}-\vec{v}$ or $\vec{v}-\vec{u}$ in $\{0,1\}^N$; and $0$ otherwise. Here the function $c$ maps a vector in $\mathbb{F}_p^N$ to its number of non-zero coordinates. Clearly $M$ is a pseudo-adjacency matrix of $G$. By Lemma \ref{lemma_eig}, for every $\vec{v}=(v_1, \cdots, v_N) \in \mathbb{F}_p^N$, $\chi_{\vec{v}}$ is an eigenvector of $M$ with eigenvalue equal to 
	\begin{align*}
	\sum_{\vec{x} \in \{0,1\}^N \setminus \vec{0}} (-1)^{c(\vec{x})} \omega^{-\langle \vec{v}, \vec{x}\rangle} + \sum_{\vec{x} \in \{0,1\}^N \setminus \vec{0}} (-1)^{-c(\vec{x})} \omega^{\langle \vec{v}, \vec{x}\rangle}
	\end{align*}
	Note that 
	\begin{align*}
	\sum_{\vec{x} \in \{0,1\}^N} (-1)^{c(\vec{x})} \omega^{-\langle \vec{v}, \vec{x}\rangle} &= \sum_{\vec{x} \in \{0,1\}^N} (-1)^{\sum_{i=1}^N x_i} \omega^{-\langle \vec{v}, \vec{x}\rangle}=\sum_{\vec{x} \in \{0,1\}^N} \prod_{i=1}^N (-1)^{x_i} \omega^{-v_ix_i}\\
	&=\prod_{i=1}^N (1-\omega^{-v_i}).
	\end{align*}
	Similarly one can show that 
	\begin{align*}
	\sum_{\vec{x} \in \{0,1\}^N} (-1)^{-c(\vec{x})} \omega^{\langle \vec{v}, \vec{x}\rangle}=\prod_{i=1}^N (1-\omega^{v_i}).
	\end{align*}
	Therefore $\chi_{\vec{v}}$ corresponds to the eigenvalue
	\begin{align*}
	\prod_{i=1}^N (1-\omega^{-v_i})+\prod_{i=1}^N (1-\omega^{v_i})-2
	\end{align*}
	When $v_j=0$ for some index $j$, $\omega^{v_j}=\omega^{-v_j}=1$, so this gives eigenvalue $-2$. Otherwise all the $v_j \in \{1, \cdots, p-1\}$. This already shows that the number of non-negative eigenvalues is at most $(p-1)^N$, and Corollary \ref{cor} gives an upper bound matching Alon's bound. But in fact we can estimate the number of non-negative eigenvalues more carefully. Note that 
	\begin{align*}
	\prod_{j=1}^N (1-\omega^{-v_j})&=\prod_{j=1}^N (1-\cos (2\pi v_j/p)+i\sin (2\pi v_j/p))\\
	&=\prod_{j=1}^N 2\sin(\pi v_j/p)\cdot e^{i (\pi/2-\pi v_j/p)}\\
	&=\left(\prod_{j=1}^N 2\sin(\pi v_j/p)\right)\cdot e^{i(\pi N/2-\pi \sum_{j=1}^N v_j/p)}
	\end{align*}
	Similarly, 
	\begin{align*}
	\prod_{j=1}^N (1-\omega^{v_j})=\left(\prod_{j=1}^N 2\sin(\pi v_j/p)\right)\cdot e^{-i(\pi N/2-\pi \sum_{j=1}^N v_j/p)}.
	\end{align*}
	Therefore
	\begin{align*}
	\prod_{i=1}^N (1-\omega^{-v_i})+\prod_{i=1}^N (1-\omega^{v_i})-2=2 \left(\prod_{j=1}^N 2\sin(\pi v_j/p)\right)\cos \left(\pi N/2-\pi \sum_{j=1}^N v_j/p \right)-2.\end{align*}
If it is non-negative, then since $\sin(\pi v_j/p) \ge 0$ for $v_j \in \{1, \cdots, p-1\}$, it must hold that  
$$\cos \left(\pi N/2-\pi \sum_{j=1}^N v_j/p \right) > 0.$$
Note that this inequality cannot hold for both $(v_1, \cdots, v_N)$ and $(u_1, \cdots, u_N)=(p-1-v_1, \cdots, p-1-v_p, p-v_{p+1}, \cdots, p-v_N)$. Since
\begin{align*}
\cos \left(\pi N/2-\pi \sum_{j=1}^N u_j/p \right)=-\cos \left(\pi N/2-\pi \sum_{j=1}^N v_j/p \right).
\end{align*}
Therefore there are at least half of those $(v_1, \cdots, v_N) \in [p-2]^p \times [p-1]^{N-p}$ correspond to negative eigenvalues. Therefore 
\begin{align*}
\alpha(G) &\le n_{\ge 0}(M) \le (p-1)^N-\frac{1}{2}(p-2)^p(p-1)^{N-p}\\
&=\left(1-\frac{1}{2}\left(1-\frac{1}{p-1}\right)^p\right)(p-1)^N.
\end{align*}
When $p \rightarrow \infty$ the constant factor tends to $1-1/(2e)$.
\end{proof}
\medskip
\noindent \textbf{Remark. }We believe that for general $p$, a more carefully analysis of the signs of these eigenvalues should show that for at most half of $(v_1, \cdots, v_N) \in [p-1]^N$, 
$\cos \left(\pi N/2-\pi \sum_{j=1}^N v_j/p \right)>0$. This would improve the constant to $1/2$. For some small values of $p$, we can actually obtain better constants. For example, when $p=3$, the same method gives $\alpha(G) \le (1/3+o(1))2^N$.  

\section{Concluding Remarks}
In the first part of this paper, we give an linear algebraic proof for Kleitman's diametric theorem, and study several extensions and generalizations. Below are some observations and related problems.

\begin{list}{\labelitemi}{\leftmargin=1em}
\item We use the Cvetkovi\'c spectral bound to find a tight upper bound for the size of a family of binary vectors with restricted diameter. For the $n$-dimensional lattice $[m]^n$ equipped with Hamming distance, Ahlswede and Khachatrian \cite{a_kh} showed that maximum family of vectors with pairwise distance at most $2d$ is attained by one of the following families, similar to their celebrated Complete Intersection Theorem. For a vector $\vec{v} \in [m]^n$, let $S_{\vec{v}}=\{i: v_i =0\}$. Then the families are
$$\mathcal{F}_i=\{\vec{v}: \vec{v} \in [m]^n, |S_{\vec{v}} \cap [n-2i]| \ge n-d-i\},$$
for $i=0, \cdots, d$. In particular, for fixed $m, d$ and sufficiently large $n$, $\mathcal{F}_0$, i.e.  the Hamming ball of radius $d$, has the maximum size among $\mathcal{F}_i$'s. It would be interesting to see whether the spectral technique we used to prove Kleitman's Theorem can also be applied in this case. Note that the proof of Theorem \ref{thm_kleitman} works almost the same if we replace the generating function $\sum_{k=2t+1}^n f(k) x^{-k}=(x+1)/(x^2-1)^{t+1}$ by the rational function $((m-1)x+1)/((x-1)((m-1)x+1))^{t+1}$. The claims regarding signs of $\lambda_i$ still work, but unfortunately it only leads to an upper bound 
$$|\mathcal{F}| \le \binom{n}{0} + (m-1)^{d-1} \binom{n}{1}+ m^2 \binom{n}{2} + \cdots ,$$
instead of the desired upper bound $$|\mathcal{F}|\le |\mathcal{F}_0|=\binom{n}{0}+(m-1)\binom{n}{1}+ \cdots + (m-1)^d \binom{n}{d}.$$  
But it is still possible that by properly choosing a generating function $f$ (or equivalently a pseudo-adjacency matrix), one can prove the tight upper bound.

\item In Section \ref{sec_ext}, we show that $f_{\{2s+1, \cdots, 2t\}}(n)$ is of the order $\Theta(n^{t-s})$, yet we only determine the exact constant factor for $s=0$ (Kleitman's Theorem), and for $t=s+1$ (Theorem \ref{thm_s+1}). We believe that the lower bound construction in Theorem \ref{thm_lower} is asymptotically best possible. Note that by our method of considering the pseudo-adjacency matrix of the form $\sum_{k} f(k) M_{n, k}$, the Cvetkovi\'c bound always gives a sum of binomial coefficients in the form of $\binom{n}{i}$. Maybe picking a more complicated pseudo-adjacency matrix so that its $(S, T)$-entry does not solely depends on $|S \cap T|$ would be useful here.

\item In Section \ref{sec_ext}, we also establish an upper bound on $f_{\mathcal{L}}(n)$ for general $\mathcal{L}$, using the number of even numbers in $\mathcal{L}$. It is not hard to see that the upper bound in Theorem \ref{thm_even} is tight up to a constant factor, for all $\mathcal{L}=\{2s+2, 2s+4, \cdots, 2t\} \cup \mathcal{L}'$, with $\mathcal{L}'$ only consisting of odd numbers. We are curious whether there are other sets $\mathcal{L}$ of allowed distances that satisfy this property. Also it would be interesting to show that the limit 
$$\lim_{n \rightarrow \infty} \frac{\log f_{\mathcal{L}}(n)}{\log n}$$
exists for all $\mathcal{L}$. And perhaps the limit must be of integer value.

\item Kleitman's diametric theorem may also be generalized in the following way. Given a connected simple graph $G=(V, E)$, the distance of two distinct vertices $u, v$ is the length of the shortest path connecting them. Let $f(G, d)$ be the maximum size of a subset of vertices with pairwise distances at most $d$. And the subset of vertices that play the role of Hamming ball is either $N_{\le d/2}(v)$, all the vertices at distance at most $d/2$ from $v$, for some $v \in V(G)$ when $d$ is even; or $N_{\le (d-1)/2}(u) \cup N_{\le (d-1)/2}(v)$ for some edge $uv \in E(G)$, when $d$ is odd. One could ask a general question: for fixed integer $d$, what graph $G$ satisfies the isodiametric inequality, i.e. $f(G, d)$ is attained by one of the sets defined above? We call such graphs $G$ {\it $d$-isodiametric}. For example, Kleitman's Theorem says that the hypercube $Q^n$ is $d$-isodiametric for $d \le n-1$. It is easy to see that a graph is $1$-isodiametric if and only if it is triangle-free. Is it possible to characterize all $d$-isodiametric graphs, or at least among Cayley graphs or vertex-transitive graphs?
\end{list}

\noindent {\bf Acknowledgments.} We would like to thank Noga Alon, Fedor Petrov and Will Sawin for several useful discussions.


\begin{thebibliography}{99}

\bibitem{acz}
R. Ahlswede, N. Cai, and Z. Zhang, 
\newblock{Diametric theorems in sequence spaces}, 
\newblock{\em Combinatorica} {\bf 12}(1) (1992), 1--17.

\bibitem{a_katona}
R. Ahlswede, G. O. H. Katona,  
\newblock{Contributions  to  the  geometry of Hamming spaces},
\newblock{\em Discrete Math} {\bf 17} (1977), 1--22.

\bibitem{a_kh}
R. Ahlswede, L. H. Khachatrian,
\newblock{The diametric theorem in Hamming spaces -- optimal anticodes},
\newblock{\em Advances in Applied Mathematics} {\bf 20} (1998), 429--449.

\bibitem{alon}
N. Alon,
\newblock{Perturbed identity matrices have high rank: proof and applications},
\newblock{\em Combin. Probab. Comput.} {\bf 18} (2009), 3–15.

\bibitem{bollobas-leader}
B.~Bollob\'as, I.~Leader,
\newblock{Maximal sets of given diameter in the grid and the torus},
\newblock{\em Discrete Mathematics} {\bf 122} (1993), 15--35.

\bibitem{clp}
E. Croot, V. Lev and P. Pach, 
\newblock{Progression-free sets in $\mathbb{Z}_4^n$ are exponentially small},
\newblock{\em Ann. of Math.} {\bf 185}(1) (2017), 331--337.

\bibitem{cvetkovic}
D. M. Cvetkovi\'c, 
\newblock{Chromatic number and the spectrum of a graph},
\newblock{\em Publ. Inst. Math. (Beograd)} {\bf 14} (28) (1972), 25--38.

\bibitem{du-kleitman}
D. Z. Du, D. J. Kleitman,
\newblock{Diameter and radius in the Manhattan metric},
\newblock{\em Journal of Discrete and Computational Geometry},
{\bf 5}(4) (1990), 351--356.


\bibitem{eg}
J.~S.~Ellenberg, D.~Gijswijt,
\newblock{On large subsets of $\mathbb{Z}_q^n$ with no three-term arithmetic progression},
\newblock{\em Ann. of Math.} {\bf 185}(1) (2017), 339--343.

\bibitem{jellenberg}
J. S. Ellenberg, 
\newblock{Difference sets missing a hamming sphere}, blog post at \\
https://quomodocumque.wordpress.com/2017/02/11/difference-sets-missing-a-hamming-sphere/.

\bibitem{erdos-hanani}
P. Erd\H os, H. Hanani,  
\newblock{On a limit theorem in combinatorial analysis},
\newblock{\em Publ. Math. Debrecen}, {\bf 10} (1963), 10--13.

\bibitem{isodiametric}
L. C. Evans, R. F. Gariepy,
\newblock{\bf  Measure theory and fine properties of functions},
CRC Press, 1991.

\bibitem{frankl-furedi}
P. Frankl, Z. Furedi, 
\newblock{The Erd\H os-Ko-Rado theorem for integer sequences},
\newblock{\em SIAM J. Algebraic Discrete Methods}, {\bf 1}(4) (1980), 316--381.

\bibitem{frankl-tokushige}
P. Frankl, N. Tokushige,
\newblock{\bf Extremal problems for finite sets}, STML 86, American Mathematical Society.

\bibitem{fw}
P.~Frankl and R. M. Wilson, 
\newblock{Intersection theorems with geometric consequences}, 
\newblock{\em Combinatorica} {\bf 1} (1981), 357--368. 

\bibitem{furstenberg1}
H.  Furstenberg,
\newblock{\textbf{Recurrence in Ergodic Theory and Combinatorial Number Theory}}, Princeton Univ. Press, 1981.

\bibitem{furstenberg2}
H. Furstenberg,
\newblock{Ergodic behavior of diagonal measures and a  theorem  of Szemer\'edi on arithmetic progressions},
\newblock{\em J. D Analyse Math}, 
{\bf 71} (1977),204--256.

\bibitem{katona}
G. O. H. Katona, 
\newblock{Intersection theorems for systems of finite sets},
\newblock{\em Acta Math. Hungar.} {\bf 15} (1964), 329--337.

\bibitem{kleitman}
D. Kleitman,
\newblock{On a combinatorial conjecture of Erd\H os},
\newblock{\em J. Combin. Theory Ser. A} {\bf 43} (1986), 85--90.

\bibitem{lth}
T. H. L\^e,
\newblock{Problems and results on intersective sets},
\newblock{Combinatorial and additive number theory},
\newblock{\em Springer Proceedings in Mathematics $\&$ Statistics} {\bf 101} (2014), 115--128. 

\bibitem{lint-wilson}
J. H. van Lint and R. M. Wilson.
\newblock{\bf A Course in Combinatorics}, Cambridge University Press, 2001.

\bibitem{ls}
L. M. Lov\'asz and L. Sauermann, 
\newblock{A lower bound for the $k$-multicolored sum-free problem in $\mathbb{Z}_m^n$},
\newblock{\em Proceedings of the London Mathematical Society}, to appear. 

\bibitem{rodl}
V. R\"odl,  
\newblock{On a packing and covering problem},
\newblock{\em European J. Combin.} {\bf 6} (1985), 69--78.

\bibitem{sarkozy}
A. S\'ark\"ozy,
\newblock{On difference sets of sequences of integers, I.},
\newblock{\em Acta Math. Acad. Sci. Hungar.}
{\bf 31} (1978), 125--149.
\end{thebibliography}
\end{document}